\documentclass[12pt]{amsart}
\usepackage{amssymb,amsmath,amsthm}
\numberwithin{equation}{section}
\newtheorem{thm}{Theorem}[section]
\newtheorem{lem}[thm]{Lemma}

\newtheorem{prop}[thm]{Proposition}

\newtheorem{defn}[thm]{Definition}

\newtheorem{prob}[thm]{Problem}

\newcommand{\weight}{\frac{\alpha !}{|\alpha|!}}

\begin{document}
\title[]{Operator-valued Herglotz kernels and functions of positive real part on the ball}
\author{Michael T. Jury}
\address{Department of Mathematics,
        University of Florida, 
        Gainesville, Florida 32603}
\email{mjury@math.ufl.edu}
\thanks{Partially supported by NSF grant DMS-0701268}
\date{\today}
\begin{abstract}  We describe several classes of holomorphic functions of positive real part on the unit ball; each is characterized by an operator-valued Herglotz formula.  Motivated by results of J.E. McCarthy and M. Putinar, we define a family of weighted Cauchy-Fantappi\`e pairings on the ball and establish duality relations between certain pairs of classes, and in particular we identify the dual of the positive Schur class.  We also establish the existence of self-dual classes with respect to this pairing, and identify some extreme points of the positive Schur class.  
\end{abstract}
\maketitle
\section{Introduction}
It is well-known that a function $f$, holomorphic in the unit disk $\mathbb D$, has nonnegative real part if and only if there exists a finite, positive Borel measure $\mu$ on the unit circle $\mathbb T$ such that
\begin{equation}\label{E:1d-herglotz}
f(z) =\int_\mathbb{T} \frac{1+z\overline{\zeta}}{1-z\overline{\zeta}}\, d\mu(\zeta) +i\Im f(0).
\end{equation}
This expression is called the \emph{Herglotz representation} of $f$.  No such simple characterization is available in higher dimensions.  One may consider the class $M^+$, consisting of all functions on the unit ball $\mathbb B^d\subset \mathbb C^d$ of the form
$$
f(z) =\int_{\partial\mathbb B^d} \frac{1+\langle z, \zeta\rangle}{1-\langle z,\zeta\rangle}\, d\mu(\zeta) +i\Im f(0)
$$
where $\mu$ is a positive measure on $\partial\mathbb B^d$, but this does not exhaust the class of holomorphic functions of positive real part.  
The present paper examines several subclasses of functions of positive real part in $\mathbb B^d$, which are shown to admit what one might call a ``noncommutative Herglotz representation;'' that is, a representation where the integral of the Herglotz kernel against a measure is replaced by the value of an operator-valued Herglotz kernel (taking values in a C*-algebra $A$) under a positive functional on the C*-algebra.   

Indeed a result of this form (though not stated in quite these terms) was established by McCarthy and Putinar \cite{MR2274973} for functions in the \emph{positive Schur class} $S^+$, which consists of all holomorphic functions $f$ on $\mathbb B^d$ for which the Hermitian kernel
$$
\frac{f(z)+\overline{f(w)}}{1-\langle z,w\rangle}
$$
is positive semidefinite.  In the noncommutative Herglotz formula for $S^+$, the algebra $C(\partial\mathbb B^d)$ of continuous functions on $\partial\mathbb B^d$ is replaced by the \emph{Cuntz algebra} $\mathcal O_d$, which is the universal C*-algebra generated by a family of isometries $V_1, \dots V_d$ satisfying the relations
$$
V^*_i V_j =\delta_{ij}, \quad \sum_{j=1}^d V_j V_j^* =I.  
$$
The Herglotz kernel is replaced by the $\mathcal O_d$-valued holomorphic function
$$
H(z,V)=2(I-\sum_{j=1}^d z_j V_j )^{-1} -I.
$$
We then have the following theorem:
\begin{thm}[\cite{MR2274973}]\label{T:mccp}
Let $f$ be holomorphic in $\mathbb B^d$.  Then $f\in S^+$ if and only if there exists a positive linear functional $\rho$ on $\mathcal O_d$ and a real number $t$ such that
\begin{equation}\label{E:mccp}
f(z)=\rho(H(z,V))+it.
\end{equation}
\end{thm}
Letting $O^+$ denote the class of all holomorphic functions with positive real part in $\mathbb B^d$, McCarthy and Putinar established the inclusions
$$
M^+ \subset S^+ \subset O^+
$$
and showed that each is proper.  They also introduced a pairing on $O^+$ such that $M^+$ and $O^+$ are dual to each other, and the dual of $S^+$ is contained in $S^+$.  The main results of this paper continue this line of investigation:  we introduce a fourth class $R^+$ which is shown to lie strictly between $M^+$ and $S^+$, and show that $R^+$ and $S^+$ are dual (under a slightly different definition of the pairing, which has more favorable topological properties).  We also establish an operator-valued Herglotz representation for $R^+$ and describe some of the extreme rays of $R^+$ and $S^+$.  
\section{Positive Classes}
Let $O$ denote the set of all holomorphic functions on the unit ball $\mathbb{B}^d\subset \mathbb{C}^d$, and $O^+ \subset O$ the functions of positive real part.  We equip $O$ with the topology of uniform convergence on compact subsets of $\mathbb{B}^d$, which makes $O$ into a locally convex topological vector space.  We will be interested in several subclasses of $O^+$ which all satisfy a number of functional-analytic conditions, codified by the following definition.  
\begin{defn}
A \emph{positive class} on $\mathbb{B}^d$ is a set of functions $\mathcal{P}\subset O^+$ which is a closed, convex cone in $O$ and is closed under dilations, that is, if $f\in \mathcal{P}$ then $f_r\in \mathcal{P}$ for all $0\leq r \leq 1$, where $f_r$ is defined by
$$
f_r(z)=f(rz)
$$
\end{defn}
Evidently $O^+$ is a positive class; we will discuss further examples below (the classes $M^+$, $R^+$ and $S^+$).  It is also clear that the intersection of any collection of positive classes is a positive class, so given any subset $E\subset O^+$ we may speak of the positive class $\mathcal P(E)$ it generates.

We introduce a family of Cauchy-Fantappi\`e pairings analogous to the classical Cauchy pairing in the disk.  In the case of the disk, if the functions are sufficiently smooth then the pairing is essentially the inner product in the Hilbert space $H^2$.  The analogous Hilbert space on $\mathbb B^d$ is the Drury-Arveson space $H^2_d$, which is most easily defined as the reproducing kernel Hilbert space whose kernel is the Fantappi\`e kernel
$$
k(z,w)=\frac{1}{1-\langle z,w\rangle} =\sum_\alpha z^\alpha \overline{w}^\alpha \frac{|\alpha| !}{\alpha !}
$$
Here we have used the usual multi-index notation:  
$$
\alpha=(\alpha_1, \dots \alpha_d)\in\mathbb N^d,
$$
$$
\alpha ! =\alpha_1 ! \alpha_2! \cdots \alpha_d!,
$$
$$
|\alpha|=\alpha_1 +\cdots +\alpha_d, 
$$
$$
z^\alpha =z_1^{\alpha_1} \cdots z_d^{\alpha_d}.
$$
For functions $f,g\in H^2_d$ with Taylor expansions
$$
f(z)=\sum_{\alpha\in\mathbb{N}^d} c_\alpha z^\alpha, \qquad  g(z)=\sum_{\alpha\in\mathbb{N}^d} d_\alpha z^\alpha
$$
the inner product is given by 
\begin{equation}\label{E: series_inner_prod}
\langle f,g\rangle_{H^2_d} = \sum_\alpha c_\alpha \overline{d_\alpha} \weight
\end{equation}
We may also express the inner product in an integral formula \cite{MR1854766}.  To state in we let $\mathcal R_d$ denote the \emph{radial derivative}:
$$
(\mathcal R_d f)(z):=\sum_{j=1}^d z_j \frac{\partial f}{\partial z_j}
$$
Then
\begin{equation}\label{E:integral_inner_prod}
\langle f,g\rangle_{H^2_d} =f(0)\overline{g(0)} + \frac{1}{d!}\int_{\mathbb B^d}(\mathcal R_df)(z)\overline{g(z)}\frac{1}{|z|^{2d}}\, d\nu(z)
\end{equation}
where $\nu$ is normalized volume measure on $\mathbb B^d$.  Note that the singularity in the integrand of (\ref{E:integral_inner_prod}) is integrable near $0$, since
$$
|(\mathcal R_d f)(z)| \leq |z| \|Df(z)\|.
$$
Now define a family of pairings $Q_r(\cdot, \cdot)$ on $O$ as follows:  given function $f,g\in O$ with Taylor expansions
$$
f(z)=\sum_{\alpha\in\mathbb{N}^d} c_\alpha z^\alpha, \qquad  g(z)=\sum_{\alpha\in\mathbb{N}^d} d_\alpha z^\alpha
$$
define for each $0< r< 1$
\begin{equation}\label{E:qr_def}
Q_r(f, g) = \sum_\alpha c_\alpha \overline{d_\alpha} r^\alpha \weight +f(0)\overline{g(0)}
\end{equation}\label{E:convergence}
\begin{lem}\label{L:convergence} The series defining $Q_r(f, g)$ converges absolutely for all $f,g\in O$ and all $r\in[0,1)$.
\end{lem}
\begin{proof}
Since the pairing $Q_r(f,g)$ is formally the same (up to the extra $f(0)\overline{g(0)}$ term) as the $H^2_d$ inner product of the dilated functions $f_{\sqrt{r}}$ and $g_{\sqrt{r}}$, by the Cauchy-Schwarz inequality it suffices to show that if $f\in O$ then all of its dilates $f_r$ belong to $H^2_d$.  But this is immediate from the integral form of the $H^2_d$ inner product (\ref{E:integral_inner_prod}), since $f_r$ is analytic across the boundary $\partial \mathbb B^d$ and so $f_r$ and its derivatives are uniformly bounded on $\overline{\mathbb B^d}$.
\end{proof}

Our interest is in the duality relationships that obtain between positive classes that admit representations by a generalized Herglotz formula, more specifically, in terms of a positive functional applied to an operator-valued Herglotz kernel, taking values in a (generally noncommutative) C*-algebra.
Using the pairings $Q_r$ we can define the dual of a subset $\mathcal{C}\subset O$:  
\begin{defn}
For $\mathcal{C}\subset O$ define
$$
\mathcal{C}^* =\{g\in O\  |\  \Re Q_r(f, g)\geq 0\text{ for all }f\in \mathcal{C}\text{ and all }r\in[0,1)  \}
$$
\end{defn}
This is slightly different from the dual $\mathcal C^\dag$ defined in \cite{MR2274973}; the advantage of the present definition is that duals of (sufficiently large) positive classes are again positive classes; in particular they are always closed and convex.

We will write $Q(f,g)$ for the pairing in the case $r=1$, provided the series is absolutely convergent.  From the proof of Lemma~\ref{L:convergence}, this will be the case whenever either $f$ or $g$ is holomorphic in a neighborhood of the closed ball.  In particular, every continuous linear functional on $O$ can be represented in this way:

\begin{thm}\label{T:duality}
A linear map $\varphi :O\to \mathbb{C}$ is continuous if and only if there exists $g\in O(\overline{\mathbb{B}^d})$ such that 
$$
\varphi(f) =Q(f,g)
$$
for all $f\in O$.  
\end{thm}
\begin{proof}  To prove that each $Q(\cdot, g)$ induces a continuous linear functional, it suffices to prove continuity at $0$.  Since $g$ is holomorphic in a neighborhood of $\overline{\mathbb B^d}$, we can choose $r<1$ so that the dilate $g_{\frac1r}$ is still holomorphic in a neighborhood of $\overline{\mathbb B^d}$.  We may then write
$$
Q(f, g)=Q(f_r, g_{\frac1r})
$$
for all $f\in O$.  If now $f_n$ is any sequence of functions tending to $0$ uniformly on compact sets, for each $r<1$ the dilates $(f_n)_r$ converge to $0$ uniformly on a neighborhood of $\overline{\mathbb{B}^d}$.  It follows that this sequence also converges to $0$ in $H^2_d$.  We then have by Cauchy-Schwarz
\begin{align*}
|Q(f_n, g)| &\leq |f_n(0)\overline{g(0)}|+|\langle (f_n)_r, g_{\frac1r}\rangle_{H^2_d}| \\
&\leq |f_n(0)\overline{g(0)}|+\|(f_n)_r\|_{H^2_d} \|g_{\frac1r} \|_{H^2_d}
\end{align*}
which converges to $0$ as $n\to \infty$.  

For the converse, let $\varphi$ be a continuous linear functional on $O$.  Let $C=C(\mathbb{B}^d)$ be the space of continuous functions on $\mathbb{B}^d$ equipped with the topology of local uniform convergence.  The dual of $C$ is the space of finite Borel measures with compact support in $\mathbb{B}^d$.  Since $O$ is closed in $C$, by Hahn-Banach $\varphi$ extends to a continuous functional on $C$.  Thus there exists a measure with compact support $K\subset \mathbb{B}^d$ such that
$$
\varphi(f)=\int_K f\, d\mu
$$
for all $f\in O$.  Now define a function $g\in O$ by
\begin{equation}\label{E:g_half_mu}
g(z)=\frac12 \int_K \frac{1+\langle z,w\rangle}{1-\langle z,w\rangle}\, d\overline{\mu}(w)
\end{equation}
Since $K$ is at positive distance from boundary $\partial\mathbb{B}^d$, $g$ is holomorphic in a neighborhood of $\overline{\mathbb{B}^d}$.  Since the power series for $(1-\langle z,w\rangle)^{-1}$ converges uniformly on $K$, we can pass the integral inside the sum to obtain the power series expansion for $g$
$$
g(z)=\frac12 \int_K \, d\overline{\mu}(w)+\sum_{\alpha\neq 0} z^\alpha \frac{|\alpha|!}{\alpha !}\int_K \overline{w}^\alpha\, d\overline{\mu}(w).
$$
It is then readily verified that
\begin{equation}\label{E:f_mu_Q_pairing}
\varphi(f)=\int_K f\, d\mu=Q(f, g).
\end{equation}
(Indeed, the pairing $Q(\cdot, \cdot)$ is defined precisely so that (\ref{E:f_mu_Q_pairing}) holds when $g$ is of the form (\ref{E:g_half_mu}).)
\end{proof}
We now introduce the positive $M^+$, which consists of functions that can be written as Herglotz transforms of positive measures on $\partial\mathbb{B}^d$:
\begin{defn}
$f\in M^+$ if and only if there exists a finite, positive Borel measure $\mu$, supported on $\partial\mathbb B^d$, such that
$$
f(z)=\int_{\partial\mathbb B^d} \frac{1+\langle z,\zeta\rangle }{1-\langle z,\zeta\rangle}\, d\mu(\zeta) +i\Im f(0)
 $$
\end{defn}
Unlike in one dimension, $M^+$ is a proper subclass of $O^+$, however we do have the following duality relation (compare \cite[Theorem 8.2]{MR2274973}):
\begin{thm}
$M^{+*}=O^+$ and $O^{+*}=M^+$.
\end{thm}
\begin{proof}
By a calculation similar to that in the proof of Theorem~\ref{T:duality}, if $g\in M^+$ and $f\in O$ we have
$$
Q_r(f, g) =\int_{\partial\mathbb{B}^d} f_r\, d\mu
$$
The real part of this is positive for all $\mu$ and all $r<1$ if and only if $f\in O^+$.  The opposite duality will follow from Theorem~\ref{T:double_dual} below.  
\end{proof}

\begin{thm}\label{T:double_dual}
If $\mathcal{P}\supset M^+$ is a positive class, then $\mathcal P^*$ is a positive class and $\mathcal P^{**}=\mathcal P$.
\end{thm}
\begin{proof}
Since $\mathcal{P}\supset M^+$, we have $\mathcal P^*\subset O^+$.  All of the remaining properties of a positive class now follow from the definition of the pairings $Q_r$.   

It is immediate from definitions that $\mathcal P\subset \mathcal P^{**}$.  The reverse inclusion is proved by a cone separation argument.  If $h\in \mathcal P^{**}\setminus \mathcal P$, since $\mathcal P$ is a closed, convex cone we can find a linear functional $\varphi$ on $O$ such that $\Re  \varphi(f)\geq 0$ for all $f\in \mathcal P$ but $\Re  \varphi(h)<0$.  By the duality theorem there exist $g\in O$ such that $\varphi(f)=Q(f,g)$ for all $f\in O$.  Since $\mathcal P$ is a positive class, we have $Q_r(f, g) =Q(f_r, g)=\varphi(f_r)$ for all $f\in\mathcal P$; taking real parts it follows that $g\in \mathcal P^*$.  But then $\Re Q_r(h, g)<0$ for $r$ sufficiently close to $1$, contradicting $h\in\mathcal P^{**}$.  Thus $\mathcal P^{**}\subset \mathcal P$.
\end{proof}
\begin{thm}\label{T:self_dual}
Suppose $\mathcal{P}\supset M^+$ is a positive class with $\mathcal{P}\subset\mathcal{P}^*$.  Then there exists a positive class $\mathcal W$ such that 
$$
\mathcal P \subset \mathcal W \subset \mathcal P^*
$$
and $\mathcal W=\mathcal W^*$.
\end{thm}
\begin{proof}
Let $\mathcal U$ denote the collection of all positive classes $\mathcal W$ such that $\mathcal P \subset \mathcal W \subset \mathcal P^*$ and $\mathcal W\subset \mathcal W^*$.  The collection $\mathcal U$ is nonempty (it contains $\mathcal P$) and partially ordered by inclusion.   We apply Zorn's lemma to obtain a maximal element:  if $\{\mathcal W_j\}_{j\in J}$ is an increasing chain in $\mathcal U$, define $\mathcal W_\infty =\overline{\cup_j \mathcal W_j}$.  It is clear that $\mathcal W_\infty$ is a closed, convex cone in $O^+$.  Moreover, if $f\in\mathcal W_\infty$ and $r<1$, we have a sequence $f_j\to f$ locally uniformly, with each $f_j$ in some $\mathcal W_j$.  Since each $\mathcal W_j$ is a positive class, $(f_j)_r\in \mathcal W_j$ and $(f_j)_r\to f_r$, so $f_r\in \mathcal W_\infty$ and hence $\mathcal W_\infty$ is a positive class.  It remains to check that $\mathcal W_\infty \subset \mathcal W_\infty^*$.  But it is straightforward to check that $\mathcal W_\infty^* =\cap_j \mathcal W_j^*$, and from this it follows that $\mathcal W_\infty$ belongs to $\mathcal U$.  Thus every increasing chain in $\mathcal U$ has an upper bound, so Zorn's lemma gives a maximal element $\mathcal W$.

To see that $\mathcal W =\mathcal W^*$, we proceed by contradiction.  If $f\in \mathcal W^*\setminus \mathcal W$, then we form a new positive class $\mathcal W_f$, the positive class generated by $\mathcal W$ and $f$.  We now check that $\mathcal W_f \subset \mathcal W_f^*$, contradicting the maximality of $\mathcal W$.  

The set of functions of the form
$$
g=g_0 + \sum_{j=1}^n c_j f_{r_j}
$$
with $g_0\in\mathcal W$, $c_j >0$ and $r_j\in[0,1)$, is dense in $\mathcal W_f$.  We check that $\Re Q_r(g,h)\geq 0$ for all $g, h\in \mathcal W_f$ of the above form; positivity for all $g,h\in \mathcal W_f$ follows by taking limits.  Letting 
$$
h=h_0 + \sum_{k=1}^m d_k f_{s_k}
$$
we have for $r<1$
\begin{align}
Q_r(g,h) &= Q_r(g_0,h_0) + \\ &+Q_r(g_0,\sum_{k=1}^m d_k f_{s_k})\\ &+Q_r( \sum_{j=1}^n c_j f_{r_j}, h_0) \\&+ Q_r(\sum_{j=1}^n c_j f_{r_j} ,\sum_{k=1}^m d_k f_{s_k} ) 
\end{align}
The first term has positive real part because $\mathcal W\subset \mathcal W^*$, the next two because $f\in \mathcal W^*$, and the last from the definition of $Q_r$.
\end{proof}
\section{Operator-valued Herglotz kernels}

A \emph{row contraction} is a $d$-tuple of operators $T=(T_1, \dots T_d)$ on a Hilbert space $\mathcal H$ such that
$$
I-T_1T_1^* -\cdots -T_dT_d^* \geq 0
$$
If $T$ is a row contraction and $z\in\mathbb{B}^d$, then the operator
$$
\langle z, T\rangle := z_1 T_1 +\cdots z_d T_d
$$
is a strict contraction.  We may then define an operator-valued Herglotz kernel
\begin{align*}
H(z, T):&= 2(I-\langle z, T\rangle )^{-1}  -I \\
&=(I+\langle z, T\rangle )(I-\langle z, T\rangle )^{-1}
\end{align*}
which we view as an analytic function on $\mathbb{B}^d$ taking values in $\mathcal B (\mathcal H)$.  A straightforward calculation shows 
\begin{align*}
\Re H(z, T)=2 (I-\langle z, T\rangle)^{-1}(I-\langle z, T\rangle \langle z, T\rangle ^*)(I-\langle z, T\rangle^*)^{-1}\geq 0
\end{align*}
Hence if $\rho$ is a positive linear functional on the C*-algebra generated by $T$, the holomorphic function
\begin{equation}\label{E:nc_herglotz}
f(z)=\rho(H(z, T))
\end{equation}
has positive real part on $\mathbb B^d$.  We will call such $f$ the \emph{Herglotz transform} of the pair $(\rho,T)$.  Theorem~\ref{T:mccp} characterizes the positive class $S^+$ in terms of Herglotz transforms.

One may observe that it is not really necessary to assume that $T$ is a row contraction for the above calculation to be valid; rather, for $H(z, T)$ to be well-defined (and have nonnegative real part) it is only necessary that 
$$
I-\langle z, T\rangle \langle z, T\rangle ^* \geq 0
$$
for all $z\in \mathbb B^d$.  This is easily seen to be equivalent to the condition that 
\begin{equation}\label{E:weak_rc}
\|\langle \zeta, T\rangle \| \leq 1 \quad \text{ for all } \zeta \in\partial\mathbb B^d.
\end{equation}
  It turns out, however, that there is no loss of generality in assuming that $T$ is a row contraction.  More precisely, we have:
\begin{thm}
If $T=T_1,\dots T_d$ is a $d$-tuple satisfying (\ref{E:weak_rc}) and $\rho$ is a positive linear functional on $C^*(T)$, then the Herglotz transform
$$
f(z)=\rho(H(z,T))
$$
belongs to $S^+$ (and is hence the Herglotz transform of a pair $(\sigma, U)$ where $U$ is a row contraction). 
\end{thm}
\begin{proof}
From the definition of $S^+$, it will suffice to prove that if $T$ satisfies (\ref{E:weak_rc}), then the operator-valued kernel
$$
\frac{H(z,T) +H(w, T)^*}{1-\langle z, w\rangle}
$$ 
is positive semidefinite on $\mathbb B^d$.  After some algebraic manipulations we find that the positivity of this kernel is equivalent to the positivity of the kernel
$$
k_T(z,w)=\frac{1-\langle z, T\rangle \langle w, T\rangle^*}{1-\langle z, w\rangle}
$$
Suppose, then, that $T_1, \dots T_d$ are operators on a Hilbert space $\mathcal H$ and 
$$
\|\sum_{j=1}^d \zeta_j T_j\|\leq 1
$$
for all $\zeta\in\partial\mathbb B^d$.  Thus for every unit vector $\xi\in\mathcal H$, we find
$$
\sum_{i=1}^d\sum_{j=1}^d \langle T_i\xi, T_j\xi\rangle_{\mathcal H} \zeta_i \overline{\zeta_j} \leq 1
$$
for all unit vectors $\zeta\in\mathbb C^d$.  In other words, the positive semidefinite $d\times d$ matrix
$$
A=(a_{ij})=\langle T_i\xi, T_j\xi\rangle
$$
is contractive on $\mathbb C^d$.  Consider now the kernel $k_T$; this operator-valued kernel is positive semidefinite if and only if the scalar-valued kernel
$$
\langle k_T(z,w)\xi, \xi\rangle_{\mathcal H} 
$$
is positive for all unit vectors $\xi\in\mathcal H$.  From the definition of the matrix $A$, this kernel is equal to 
$$
\frac{1-\langle A^{1/2}z, A^{1/2}w\rangle }{1-\langle z, w\rangle}
$$
which is positive since $\|A^{1/2}\|\leq 1$ on $\mathbb C^d$.
\end{proof}
Three particular row contractions (and the C*-algebras they generate) will be of interest.  First, we may view the commutative C*-algebra $C(\partial\mathbb{B}^d)$ as the universal C*-algebra generated by a $d$-tuple of commuting normal elements $Z=(Z_1, \dots Z_d)$ satisfying
$$
I-Z_1 Z_1^*-\cdots -Z_dZ_d^* =0
$$
(these are called \emph{spherical operators} in \cite{MR1668582}; one should think of these ``universal generators'' simply as the coordinate functions $\zeta_1, \dots \zeta_d$ in $C(\partial\mathbb B^d)$).  Second, we consider the row contraction $S=(S_1, \dots S_d)$ where $S_j $ is multiplication by $z_j$ on the space $H^2_d$.  These generate the \emph{$d$-variable Toeplitz algebra} $\mathcal T_d$.  Finally, we have the $d$-tuple of \emph{Cuntz isometries}, which are isometries with orthogonal ranges whose final projections sum to $I$.  That is, $d$-tuples $V=(V_1, \dots V_d)$ such that
$$
V_i^* V_j =\delta_{ij}, \qquad I-V_1V_1^* - \cdots -V_dV_d^* =0.
$$
The \emph{Cuntz C* algebra} $\mathcal O_d$ is the universal C*-algebra generated by the Cuntz isometries.
We now have three positive classes $M^+, R^+$ and $S^+$ which can be expressed in terms of operator-valued Herglotz formulae.  The classes $M^+$ and $S^+$ were described in \cite{MR2274973}.
\begin{thm}\label{T:herglotz_formulas}
Each of the following is a positive class on $\mathbb B^d$:
\begin{itemize} 
\item $M^+$, consisting of all functions of the form
$$
f(z)=\rho(H(z, Z))+i\Im f(0)
$$
where $\rho$ is a positive linear functional on $C(\partial\mathbb{B}^d)$,

\item $R^+$, consisting of all functions of the form
$$
f(z)=\rho(H(z, S))+i\Im f(0)
$$
where $\rho$ is a positive linear functional on $\mathcal T_d$,

\item $S^+$, consisting of all functions of the form
$$
f(z)=\rho(H(z, V))+i\Im f(0)
$$
where $\rho$ is a positive linear functional on $\mathcal O_d$.
\end{itemize}
\end{thm}
\begin{proof}
We prove that $R^+$ is a positive class; it will be clear that an analogous proof works for the others.

It is clear from the definition that $R^+$ is a convex cone, since the set of positive functionals on a C*-algebra is a convex cone.  To see that $R^+$ is closed, let $f_n$ be a sequence in $R^+$ converging uniformly on compact sets to a function $f\in O^+$ (we assume for simplicity that $f_n(0)$ is real for all $n$).  Then there is a sequence of positive functionals $\rho_n$ on $\mathcal T_d$ such that
$$
f_n(z) = \rho_n(H(z, S)).
$$
Since the sequence $\|\rho_n\|=|\rho_n(I)| =|f_n(0)|$ is bounded, by the Banach-Alaoglu theorem $\rho_n$ has a weak-* cluster point $\rho$.  But then we have
$$
f(z) = \rho(H(z, S))
$$
and hence $f\in R^+$.  

To prove that $R^+$ is dilation invariant, we obtain an equivalent definition of $R^+$ (which will also be of use later).  By the Gelfand-Naimark-Segal theorem applied to the functional $\rho$, there exist a Hilbert space $\mathcal H$, a vector $\xi\in \mathcal H$ and a representation $\pi :\mathcal T_d\to \mathcal B(\mathcal H)$ such that
\begin{equation*}
\rho(H(z,S))=\langle H(z, \pi(S))\xi, \xi\rangle_{\mathcal H}
\end{equation*}
By Arveson's dilation theorem, we conclude that $f\in R^+$ if and only if there exists a Hilbert space $\mathcal H$, a vector $\xi\in\mathcal H$ and a commuting row contraction $T$ such that (ignoring the harmless imaginary constant)
\begin{equation*}
f(z)=\langle H(z, T)\xi, \xi\rangle_{\mathcal H}
\end{equation*}
Representing $f\in R^+$ in this way, we have
\begin{align}
f_r(z)&=\langle H(rz, T)\xi, \xi\rangle_{\mathcal H} \\
      &=\langle H(z, rT)\xi, \xi\rangle_{\mathcal H} 
\end{align}
and hence $f_r\in R^+$ since $rT$ is a commuting row contraction.
\end{proof}

Of course, since a positive linear functional on $C(\partial\mathbb B^d)$ is just a finite Borel measure $\mu$ supported on $\partial\mathbb{B}^d$, the Herglotz formula for $M^+$ takes the form
$$
f(z)=\int_{\partial\mathbb B^d} \frac{1+\langle z, \zeta\rangle}{1-\langle z,\zeta \rangle}\, d\mu(\zeta) +i\Im f(0)
$$
in accord with the original definition.

As in the proof of the theorem, by combining the GNS representation theorem with appropriate dilation theorems we can obtain equivalent descriptions of $R^+$ and $S^+$ that are sometimes useful.  We first recall that the original definition of $S^+$ is the set of all $f\in O$ such that the kernel
\begin{equation}\label{E:splus_def}
\frac{f(z)+\overline{f(w)}}{1-\langle z,w\rangle}
\end{equation}
is positive.  In \cite{MR2274973} it is shown that $f\in S^+$ if and only if there exists a Cuntz isometry $V$ acting on a space $\mathcal H$ and a vector $\xi\in \mathcal H$ such that
$$
f(z)=\langle H(z,V)\xi, \xi\rangle +\Im f(0).
$$
By Popescu's dilation theorem \cite{MR972704}, every row contraction $T$ dilates to a Cuntz isometry, and hence $S^+$ consists of all functions of the form
\begin{equation}\label{E:T_herglotz}
f(z)=\langle H(z,T)\xi, \xi\rangle +\Im f(0).
\end{equation}
where $T$ is any row contraction.  In contrast, the Arveson dilation theorem \cite{MR1668582} says that every commuting row contraction dilates to a representation of the $d$-shift $S$, and hence $R^+$ consists of all functions $f$ admitting a representation of the form (\ref{E:T_herglotz}) for some \emph{commuting} row contraction $T$.  

The formalism established thus far leads to an ``abstract Herglotz formula'' characterizing positive classes $\mathcal P$ contained in $S^+$:
\begin{thm}
To each positive class $\mathcal P\subset S^+$ on $\mathbb B^d$ there corresponds a weak-* closed, convex subset $\widehat{\mathcal P}$ of the state space of the Cuntz algebra $\mathcal O_d$, such that $f\in\mathcal P$ if and only if there exists a state $\rho\in \widehat{\mathcal P}$ and a real number $\lambda \geq 0$ such that
\begin{equation}\label{E:abstract_herglotz}
f(z)=\lambda\rho(H(z, V))+\Im f(0)
\end{equation}
\end{thm}
\begin{proof}
Given $\mathcal P$, let $\widehat{\mathcal P}$ consist of all states representing functions $f\in\mathcal P$ as in (\ref{E:abstract_herglotz}).  Since $\mathcal P$ is a positive class, $\widehat{\mathcal P}$ is convex, and it is straightforward to check that if $\rho_n\to \rho$ weak-*, then $f_n\to f$ uniformly on compact sets, so $\widehat{\mathcal P}$ is closed.
\end{proof}
It must be kept in mind, however, that in general the state $\rho$ corresponding to $f$ may be far from unique.  This is even the case in $M^+$, since the closed linear span of the Herglotz kernels and their adjoints is a proper subspace of $C(\partial\mathbb B^d)$.  

We now return to the particular positive classes described by Theorem~\ref{T:herglotz_formulas}.  Our main goal is to establish the duality between $R^+$ and $S^+$, resolving an open question of \cite{MR2274973}.  

To begin with we recall the ``symmetrized functional calculus'' from \cite{MR2274973}.  If $T=(T_1, \dots T_d)$ is a $d$-tuple of operators, we can make sense of a monomial $z^\alpha$ in $T$ by averaging over all possible orderings of the $T_i$'s.  Explicitly, for a multi-index $\alpha=(\alpha_1, \dots \alpha_d)$ we define
$$
(z^\alpha)^{sym}(T)=\weight \sum T_{i_1} T_{i_2}\cdots T_{i_{|\alpha|}}
$$
where the sum is taken over all permutations of $\alpha_1$ 1's, $\alpha_2$ 2's, etc.  

We require one more definition:  if $f\in O$, define
$$
\check f (z):=\overline{f(\overline{z})}
$$
If $E\subset O$, then let $\check E$ denote the functions $f\in O$ such that $\check f\in E$.  Finally we observe from the definition (\ref{E:splus_def}) of $S^+$  that $\check{S}^+=S^+$.
The following theorem identifies the dual of $S^+$, answering a question posed by McCarthy and Putinar \cite{MR2274973}.
\begin{thm}\label{T:rs_duality}
$R^{+*}=S^+$ and $S^{+*}=R^+$.
\end{thm}
\begin{proof}
By the above discussion, $g$ belongs to $R^+$ if and only if there exists a commuting row contraction $T$ on a Hilbert space $\mathcal H$ and a vector $\xi\in\mathcal H$ such that
\begin{align}
\frac12 [g(z)+\overline{g(0)}] &= \langle (I-\langle z, T\rangle)^{-1} \xi, \xi\rangle \\
                               &= \sum_{j=0}^\infty \langle \langle z,T\rangle^j \xi,\xi\rangle \\
                               &= \sum_\alpha z^\alpha \langle (z^\alpha)^{sym}(T)\xi, \xi\rangle \frac{|\alpha|!}{\alpha !}
\end{align}
It follows that for any $f\in O$ and $r<1$,
$$
Q_r(f,g) =2\langle (\check{f}_r)^{sym}(T)\xi, \xi\rangle =2\langle \check{f}_r(T)\xi, \xi\rangle
$$
since $T$ commutes.  So $\Re Q_r(\check{f},g)\geq 0$ for all $g\in R^+$ if and only if $\Re \check{f}_r(T)\geq 0$ for all commuting row contractions $T$.  Taking $T=S$ shows that we must have $\check{f}_r\in S^+$ for all $r<1$, so $f_r\in S^+$ for all $r<1$ and hence $f\in S^+$.  But then $\Re \check{f}_r(T)\geq 0$ for all commuting row contractions $T$ by von Neumann's inequality for $S^+$.  Thus, $R^{+*}=S^+$.  The opposite duality follows from Theorem~\ref{T:double_dual}.
\end{proof}

\begin{thm}
$R^+$ is a proper subset of $S^+$.
\end{thm}
\begin{proof}
Since $R^+$ is the dual of $S^+$, we show that $S^{+*}\subsetneq S^+$.  Using the description of $S^+$ given in Theorem~\ref{T:herglotz_formulas}, by a calculation similar to that in the proof of Theorem~\ref{T:rs_duality} we see that $f\in S^{+*}$ if and only if 
$$
\Re \langle f_r^{sym}(V)\xi, \xi\rangle \geq 0
$$
for all Cuntz isometries $V$.  Suppose by way of contradiction that this is the case for all $f\in S^+$.  Let $\mathcal A$ denote the set of all polynomials $p(S)\in \mathcal T_d$ and let $\mathcal S$ be the operator system $\mathcal A +\mathcal A^*$.  Our assumption would then imply that for every row isometry $V$ acting on a Hilbert space $\mathcal H$, the map from $\mathcal S$ to $\mathcal B(\mathcal H)$ given by
$$
p(S) +q(S)^* \to p^{sym}(V) +q^{sym}(V)^*
$$
is positive.  Then by \cite[Corollary 2.8]{MR1976867}, the map sending $p(S)$ to $p^{sym}(V)$ is contractive.  But Davidson and Pitts \cite[Example 3.4]{MR1627901} give an explicit counterexample showing that this map is not contractive:  in particular, for the polynomial $p(z)=z_1 +z_1z_2$ we have $\|p(S)\|<\sqrt{2}$ but $\|p^{sym}(V)\|=\frac52$.  Thus $R^+=S^{+*}\subsetneq S^+$.  
\end{proof}
\begin{thm}
The following chain of inclusions holds, and each is proper:
$$
M^+\subset R^+\subset S^+\subset O^+.
$$
\end{thm}
\begin{proof}
The inclusion of $R^+$ in $S^+$ is given by the previous theorem.  The proper inclusion of $S^+$ in $O^+$ is well known, e.g. proved in \cite{MR2274973}.  By Theorem~\ref{T:rs_duality} the inclusion $M^+\subsetneq  R^+$ is dual to $S^+\subsetneq O^+$.  
\end{proof}
Theorem~\ref{T:self_dual} guarantees the existence of a self-dual positive class lying between $R^+$ and $S^+$, however we do not know any concrete example.
\begin{prob}
Give an explicit example of a self-dual positive class lying between $R^+$ and $S^+$.  
\end{prob}
\section{Extreme points}
If $\mathcal P$ is a positive class, we can normalize each function $f\in\mathcal P$ so that $f(0)=1$; we denote the resulting subclass $\mathcal P_0$.  Since $O^+_0$ is compact, it follows that $\mathcal P_0$ is also compact and hence by the Krein-Milman theorem it is the closed convex hull of its extreme points.  It would obviously be desirable to identify the extreme points of the classes $R^+_0$ and $S^+_0$.  It is not hard to see that the extreme points of $M^+_0$ are exactly the functions obtained as Herglotz integrals of point masses; that is the extreme points of $M^+_0$ are the functions
$$
h_\zeta(z) =\frac{1+\langle z, \zeta\rangle}{1-\langle z, \zeta\rangle},  \quad \zeta\in\partial\mathbb B^d.
$$
It is known that these functions are not extreme for $O^+$; for example when $d=2$ the functions
$$
\frac{1+z_1}{1-z_1}+z_2^2 h(z_1, z_2)
$$
belong to $O^+$ whenever $|h|\leq \frac14$ \cite[Section 19.2.6]{MR601594}.
We will show presently that the functions $h_\zeta$ are extreme points for $S^+$; but of course in $S^+$ there must be others as well.  As a small step towards finding them, we have the following.
\begin{prop}\label{P:pure_states}
If $\mathcal P_0$ is a normalized positive class lying between $M^+_0$ and $S^+_0$ and $f$ is an extreme point of $\mathcal P_0$, then there exists a pure state on the Cuntz algebra $\mathcal O_d$ such that
$$
f(z)=\rho (H(z, V))
$$ 
\end{prop}
\begin{proof}
The proof is a routine application of the Krein-Milman theorem:  let $\mathcal S$ denote the set of all states $\rho$ on $\mathcal O_d$ such that $\rho(H(z,V)$ belongs to $\mathcal P_0$.  $\mathcal S$ is a weak-* compact, convex set in the dual of $\mathcal O_d$.  Let $\mathcal S_f \subset \mathcal S$ denote those states such that $\rho(H(z,V))=f$.  $\mathcal S_f$ is again compact and convex, so by Krein-Milman is the closed hull of its extreme points.  It is now not hard to see that if $f$ is extreme then every extreme point of $\mathcal S_f$ is in fact extreme in $\mathcal S$.  Indeed, suppose $f$ is extreme and $\rho$ is extreme in $\mathcal S_f$.  If $\rho_0, \rho_1\in \mathcal S$ satisfy $t\rho_0+(1-t)\rho_1 =\rho$ for some $t\in(0,1)$, then taking Herglotz transforms shows $\rho_0, \rho_1\in\mathcal S_f$ and hence $\rho_0=\rho_1=\rho$.
\end{proof}
Of course, the converse does not hold; there are many examples of pure states $\rho$ for which the associated $f$ is not extreme in $S^+_0$.  

To prove that the functions $h_\zeta $ are extreme for $S^+$, we make use of the following strengthened Schwarz lemma for the Schur class $\mathcal S_d$.  Recall that a function $\varphi$ belongs to $\mathcal S_d$ if and only if the Hermitian kernel
$$
k^\varphi(z,w) =\frac{1-\varphi(z)\overline{\varphi(w)}}{1-\langle z,w\rangle}
$$
is positive semidefinite.  A simple calculation shows that $\varphi\in\mathcal S_d$ if and only if its Cayley transform $\frac{1+\varphi}{1-\varphi}$ belongs to $S^+$.  
\begin{lem}
Suppose $g\in\mathcal S_d$, $g(0)=0$ and
$$
\frac{\partial g}{\partial z_1}(0) =1.
$$ 
Then $g(z)=z_1$.  
\end{lem}
\begin{proof}
To say that $g\in\mathcal S_d$ means that the kernel
\begin{equation}\label{E:g_pick}
\frac{1-g(z)\overline{g(w)}}{1-\langle z, w\rangle}
\end{equation}
is positive.  Let $G$ denote the restriction of $g$ to the disk $D=\{(z_1, 0):|z_1|<1 \}$.  Then $G(0)=0$, $G^\prime(0)=1$ and $|G(z)|\leq 1$ for all $|z|<1$.  By the one-variable Schwarz lemma, we have $G(z_1)=z_1$.  Thus the restriction of $g$ to $D$ equals $z_1$.  But since restriction of the kernel (\ref{E:g_pick}) to $D$ is 
$$
\frac{1-G(z_1)\overline{G(w_1)}}{1-z_1 \overline{w_1}}\equiv 1
$$
which trivially has finite rank, the function $g|_D$ has a unique extension to an $\mathcal S_d$ function in $\mathbb B^d$.  This forces $g\equiv z_1$.
\end{proof}
\begin{thm}
For each $\zeta\in\partial\mathbb B^d$, the function
$$
h_\zeta(z) =\frac{1+\langle z, \zeta\rangle}{1-\langle z, \zeta\rangle}
$$
is an extreme point of $S^+$.  
\end{thm}
\begin{proof}
Without loss of generality we assume $\zeta =e_1=(1,\dots 0)$.  Suppose there exist functions $f, g$ in the Schur class $\mathcal S_d$ and $t\in (0,1)$ such that
$$
h(z)=h_{e_1}(z)=\frac{1+z_1}{1-z_1} =t\frac{1+f(z)}{1-f(z)}  +(1-t)\frac{1+g(z)}{1-g(z)}
$$
Restrict this equation to the disk $D=\{(z_1, 0):|z_1|<1\}$.  Since $h(z)=h(z_1)$ is extreme for $O^+(\mathbb{D})$, it follows that $f$ and $g$ must equal $z_1$ when restricted to $D$.  By the uniqueness in the Schwarz lemma for $\mathcal S_d$, we must have $f(z)=g(z)=z_1$ for all $z\in\mathbb B^d$.  Thus $h$ is extreme in $S^+$.
\end{proof}

By Proposition~\ref{P:pure_states}, there exists a pure state $\rho$ on $\mathcal O_d$ such that
$$
h_\zeta(z)=\rho (H(z, V)).
$$
In fact it can be shown that in this case the state $\rho$ is unique, and equal to the \emph{Cuntz state} $\omega_\zeta$ defined by
$$
\omega_\zeta (V_{i_1}\cdots V_{i_m}V^*_{j_1}\cdots V^*_{j_n})=\zeta_{i_1} \cdots \zeta_{i_m} \overline{\zeta_{j_1}} \cdots \overline{\zeta_{j_n}}.
$$



It seems unlikely that there would be any very concrete parametrization of the extreme points of $S^+_0$.  Nonetheless at present we do not know any extreme points of $S^+_0$ besides the functions $h_\zeta$, so it would be desirable to find other examples.  

In the case of $R^+_0$,  the $d$-variable Toeplitz algebra $\mathcal T_d$ is a type I C*-algebra so there is more hope of classifying the extreme points.  However there is still the difficulty of the non-uniqueness of the Herglotz representation.  The argument in the proof of the previous proposition shows that every extreme point of $R^+_0$ comes from a pure state of $\mathcal T_d$, however as in the case of $S_0^+$ simple examples show that the converse does not hold.  
\begin{prob}
Find other extreme points of $S_0^+$ and $R_0^+$.  Can the extreme points of $R^+_0$ be classified?  Which pure states on $\mathcal T_d$ (resp. $\mathcal O_d$) give rise to extreme points?
\end{prob}

We conclude with a remark on the growth of functions in $S^+$.  One reason for historical interest in extreme points of $O^+$ was their connection with the inner function conjecture (namely, that there were no inner functions on the ball $\mathbb B^d$ when $d>1$); in fact it had been conjectured that functions $f\in O^+$ all belonged to the Hardy space $H^1$ (see \cite[Chapter 19]{MR601594}).  The inner function conjecture turned out to be false, but it is obviously of interest to ask the same questions of other classes, particularly $S^+$.  Recall that for $0<p<\infty$, a function $f\in O$ belongs to the Hardy space $H^p$ if and only if
\begin{equation}\label{E:hp}
\sup_{0<r<1}\int_{\partial\mathbb B^d} |f(r\zeta)|^p \, d\sigma(\zeta) <\infty,
\end{equation}
where $\sigma$ denotes surface measure on the sphere.  When $p\geq 1$ the $p^{th}$ root of this quantity defines a norm on $H^p$.  By standard estimates, the extreme points $h_\zeta$ belong to $H^p$ for all $p<\frac{d+1}{2}$.  In fact this is characteristic of the growth of functions in $S^+$:

\begin{thm}\label{T:Hp}
$S^+(\mathbb B^d) \subset H^p(\mathbb B^d)$ for all $p<\frac{d+1}{2}$. 
\end{thm}
\begin{proof}
Let $g\in S^+$ and write $g=\frac{1+\varphi}{1-\varphi}$ for $\varphi\in\mathcal S_d$.  Define a holomorphic mapping $\psi:\mathbb B^d\to \mathbb B^d$ by
$$
\psi(z)=(\varphi(z), 0\dots 0)
$$
Since $\varphi\in\mathcal S_d$, the kernel
$$
\frac{1-\langle \psi(z),\psi(w)\rangle}{1-\langle z, w\rangle}=\frac{1-\varphi(z)\overline{\varphi(w)}}{1-\langle z, w\rangle}
$$
is positive semidefinite.  It follows that the composition operator
$$
C_\psi :f\to f\circ\psi
$$
is bounded on $H^2$ \cite[Theorem 5]{mj-comp}, and hence bounded on $H^p$ for all $p<\infty$ \cite[Corollary 1.2]{MR783578}.  Thus
$$
g=\frac{1+\varphi}{1-\varphi} =C_\psi \left(\frac{1+z_1}{1-z_1}\right)
$$
belongs to $H^p$ for all $p<\frac{d+1}{2}$.
\end{proof}
In contrast, $O^+$ contains functions which do not belong to $H^1$; indeed if $\varphi$ is an inner function then $g=\frac{1+\varphi}{1-\varphi}\notin H^1$.  As a corollary we see that the Schur class $\mathcal S_d$ does not contain any inner functions when $d>1$.

\bibliographystyle{plain} 
\bibliography{realpart} 
\end{document}